\documentclass[12pt]{article}

\usepackage{calc}
\usepackage[top=1in,bottom=1in,left=.75in,right=.75in,includehead]{geometry}

\usepackage[dvipsnames,table]{xcolor}
\usepackage[colorlinks=true,citecolor=RoyalBlue,linkcolor=Red]{hyperref}

\usepackage{amsmath,amsthm,amssymb,enumerate}
\usepackage{graphicx}
\usepackage{caption}
\usepackage{float, subcaption, tikz}
\usetikzlibrary{decorations.pathreplacing,calc}
\tikzstyle{b}=[thick,decorate, decoration={brace,amplitude=5pt,mirror}, xshift=0.4pt,yshift=-0.4pt]
\tikzstyle{l}=[midway,yshift=-0.5cm]
\newtheorem{theorem}{Theorem}
\newtheorem{lemma}[theorem]{Lemma}
\newtheorem{cor}[theorem]{Corollary}
\newtheorem{prop}[theorem]{Proposition}
\theoremstyle{definition}

\renewcommand{\S}{\mathcal{S}}

\DeclareMathOperator{\des}{des}
\DeclareMathOperator{\Des}{Des}
\newcommand{\Inv}{\mathcal{I}}

\newcommand{\x}{\boldsymbol{x}}
\usepackage{authblk}
\title{Involutions and the Gelfand character}

\author[1]{Kassie Archer\footnote{email: karcher@uttyler.edu}}
\author[2]{Virginia Germany}
\author[3]{C. Marin King}
\author[4]{L.-K. Lauderdale}
\date{}
\affil[1]{\footnotesize{Department of Mathematics, University of Texas at Tyler, Tyler, TX 75799}}
\affil[2]{\footnotesize{Mathematics Department, Kilgore College, Kilgore, TX 75662}}
\affil[3]{\footnotesize{Department of Mathematics, University of Missouri, Columbia, MO 65211}}
\affil[4]{\footnotesize{Department of Mathematics, Towson University, Towson, MD 21252}}

\usepackage[backend=bibtex]{biblatex}
\begin{document}

\maketitle

\abstract{
The Gelfand representation of $\mathcal{S}_n$ is the multiplicity-free direct sum of the irreducible representations of $\mathcal{S}_n$. In this paper, we use a result of Adin, Postnikov, and Roichman to find a generating function for the Gelfand character. In order to find this generating function, we investigate descents of so-called $\lambda$-unimodal involutions.
}

\paragraph{Keywords:} involutions, enumeration, $\lambda$-unimodal permutations, descents, Gelfand character


\section{Introduction}

A permutation is \textit{unimodal} provided its one-line notation is increasing, then decreasing.  Given any composition $\lambda$ of the positive integer $n$, we say that a permutation is $\lambda$\textit{-unimodal} if it is comprised of contiguous unimodal segments whose lengths are determined by $\lambda$.  These $\lambda$-unimodal permutations (called $\mu$-unimodal in \cite{AR2015})  are the topic of research by numerous authors. They appeared in \cite{APR08, AR2015, AS18, A2015, ER2016, ER2014,R1998,R1997}, among others, as a tool to study characters of the symmetric group, Schur-positivity, other aspects of quasi-symmetric functions, and some enumerative applications; some of these applications are discussed in more detail below.  

In \cite{AR2015}, the authors describe several characters that can be written as the sum over a set of these permutations with a given property. Specifically, if $\mathcal{B}_n \subseteq \S_n$ is a so-called \textit{fine set}, then letting $\S^\lambda$ denote the set of $\lambda$-unimodal permutations and letting $\des_\lambda(\pi)$ denote the number of $\lambda$-descents of a permutation $\pi$ (defined in Section \ref{sec:background}), we have
$$
\chi_\lambda = \sum_{\pi\in\S^\lambda\cap\mathcal{B}_n} (-1)^{\des_\lambda(\pi)}
$$
for some character $\chi$ of the symmetric group. For example, if $\mathcal{B}_n=\S_n$, then $\chi$ is the regular representation and if $\mathcal{B}_n$ is the set of involutions in $\S_n$, $\chi$ is the Gelfand character. Other fine sets include conjugacy classes in $\S_n$ and their unions. 
In \cite{A2016}, the first author of this article investigated $\lambda$-unimodal cycles and their relationship to a specific induced character of the symmetric group using methods developed from studying the periodic orbits of certain dynamical systems. This proved to have some interesting applications to enumerative combinatorics, allowing us to recover a result that appeared in \cite{Elizalde11}, which states that the number of permutations in $\S_{n-1}$ with descent set $D \subseteq [n - 2]$ is equal to the number of cyclic permutations whose descent set is either $D$ or $D\cup \{n -1\}$. 

Though there is much literature involving $\lambda$-unimodal permutations, they are not often studied as purely combinatorial objects. The enumeration of these permutations with respect to certain statistics or with certain properties remains mostly open. However, their enumeration and relationship to other combinatorial objects has been shown to have interesting implications. For example, in \cite{A2015}, Athanasiadis used a relationship between these permutations and certain graphs in order to prove a conjecture of Shareshian and Wachs; this conjecture gives a refinement of Stanley's chromatic symmetric function for graphs. In \cite{AS18}, Alexandersson and Sulzgruber gave several results regarding $p$-positivity and combinatorial interpretations for coeffieicents of the power sum expansion of several families of quasi-symmetric functions using combinatorial results about $\lambda$-unimodal permutations. These applications gives us a reason to consider the enumeration of these permutations with respect to other properties and statistics.

In this paper, we investigate $\lambda$-unimodal involutions, i.e., those $\lambda$-unimodal permutations that are their own algebraic inverse, and we use them to compute the Gelfand character.
In  \cite{APR08, AR2015}, it is shown that these involutions have a direct relationship to the Gelfand character, $\chi^G$, which is the character associated to the representation of $\S_n$ obtained by taking the multiplicity-free direct sum of the irreducible representations of $\S_n$. For example, see \cite{APR08}. 
Specifically, if $\Inv^\lambda$ denotes the set of $\lambda$-unimodal involutions and $\des_\lambda(\pi)$ denotes the number of $\lambda$-descents of a permutation $\pi$ (defined in Section \ref{sec:background}), then 
\begin{equation}\label{eq:char}
\chi_\lambda^G = \sum_{\pi\in\Inv^\lambda}(-1)^{\des_\lambda(\pi)}.
\end{equation}


The bulk of this paper is dedicated to enumerating $\lambda$-unimodal involutions via a recursive generating function. This can be further refined to a generating function for the number of $\lambda$-unimodal involutions with a given number of $\lambda$-descents, which in turn gives a generating function for the Gelfand character (see Theorem \ref{main theorem 3} and Corollary \ref{cor:Gelfand}). This yields a new way of computing the Gelfand character different than the ones currently known; see \cite{IRS90, M13, M-1937, N1940, B2004, R1997}.
This also gives us an approach to address an open question in \cite{R2014}, in which Roichman comments on the desirability of combinatorial proofs to certain character formulas, such as Equation~\eqref{eq:char}. 

In addition, we provide a combinatorial proof for a formula for the character of the regular representation of $\S_n$. This formula appears in {\cite[Cor.\ 3.8]{AR2015}}. In Section \ref{sec: regular}, we prove this character formula by showing that for $\lambda \vDash n$,
\begin{equation}\label{eq:reg char}
\sum_{\pi\in\S^\lambda}(-1)^{\des_\lambda(\pi)} = \begin{cases}n! & \text{ if $\lambda= (1,1,\ldots,1)$} \\ 0 & \text{ otherwise,}\end{cases}
\end{equation}
which coincides with the regular representation.

\section{Background and Notation}\label{sec:background}

Let $\S_n$ be the set of permutations on $[n] = \{1,2,\ldots,n\}$, and write $\pi \in \S_n$ in its one-line notation as $\pi=\pi_1\pi_2\ldots\pi_n=\pi(1)\pi(2)\ldots \pi(n)$.  A permutation $\pi \in \S_n$ is \textit{unimodal} if there exists $i \in [n]$ such that 
	\[\pi_1<\pi_2<\cdots<\pi_{i-1}<\pi_i>\pi_{i+1}>\cdots>\pi_{n-1}>\pi_n;\]
that is, $\pi$ is increasing then decreasing. Similarly, any sequence or segment is unimodal if it is increasing, then decreasing.  A \textit{composition} of the integer $n$, denoted $\lambda \vDash n$,  is a sequence of positive integers $\lambda=(\lambda_1,\lambda_2, \ldots, \lambda_k)$ such that $\sum \lambda_i=n$.  Given a composition $\lambda=(\lambda_1,\lambda_2, \ldots, \lambda_k)\vDash n$, we say that $\pi \in \S_n$ is $\lambda$\textit{-unimodal} provided $\pi$ is composed of $k$ contiguous segments, where the $i$-th segment is unimodal of length $\lambda_i$ with $i\in[k]$.  For example, the permutation $\pi = 139654872 \in \S_9$ is $(5,4)$-unimodal because the first five entries $13965$ and the last four entries $4872$ both form unimodal segments of $\pi$; the pictorial representation of this permutation can be seen in Figure~\ref{fig:EXAMPLES}(a).  

The permutation $\pi \in \S_n$ has a \textit{descent} at position $i$ if $\pi_i>\pi_{i+1}$. The \textit{descent set} of $\pi$, denoted $\Des(\pi)$, is the set of descents of $\pi$ and the \textit{descent number} of $\pi$, denoted $\des(\pi)$, is the number of descents of $\pi$.  If $\lambda=(\lambda_1,\lambda_2, \ldots, \lambda_k) \vDash n$, we say that $i$ is a $\lambda$\textit{-descent} of $\pi$ if $i$ is a descent of $\pi$ that is within one of the segments corresponding to $\lambda$. In other words, we define the set of $\lambda$-descents of $\pi$, denoted $\Des_\lambda(\pi)$, to be the set $$\Des_\lambda(\pi)=\Des(\pi)\setminus\{\lambda_1, \lambda_1+\lambda_2, \ldots, \lambda_1+\lambda_2+\cdots+\lambda_{k-1}\}.$$ 
We let $\des_\lambda(\pi)$ denote the number of $\lambda$-descents of $\pi$.  For example, if $\lambda=(5,4)$ and $\pi = 129654873 \in \S_9$, then we have $\des(\pi)=5$ and $\des_\lambda(\pi)=4$. In this example, the descent at position 5 is not a $\lambda$-descent.  

We say $\pi \in \S_n$ is an \textit{involution} if it is its own inverse, i.e.,  $\pi=\pi^{-1}$.  Equivalently, every involution is comprised of only disjoint transpositions and fixed points, and in its pictorial representation, every involution is symmetric about the diagonal.  The $(4,3,2)$-unimodal involution $476183259 \in \S_9$ is depicted in Figure~\ref{fig:EXAMPLES}(b).
Additionally, let $\S^\lambda$ denote the set of $\lambda$-unimodal permutations, and let $\Inv^\lambda$ denote the set of $\lambda$-unimodal involutions. For example, $139654872\in S^{(5,4)}$ and $476183259 \in \Inv^{(4,3,2)}$.

  	\begin{figure}[H]
	\centering
		\begin{tabular}{ccc}
			\subcaptionbox{$(5,4)$-unimodal permutation $\pi = 139654872 \in \S_9$}{
			\begin{tikzpicture}[scale=.5]
				\newcommand\myx[1][(0,0)]{\pic at #1 {myx};}
				\tikzset{myx/.pic = {\draw 
				      (-2.5mm,-2.5mm) -- (2.5mm,2.5mm)
				      (-2.5mm,2.5mm) -- (2.5mm,-2.5mm);}}

				\draw[gray] (0,0) grid (9,9);
				\draw[thick] (5,0) -- (5,9);
				\foreach \x/\y in {1/1,2/3,3/9,4/6,5/5,6/4,7/8,8/7,9/2}
				\myx[(\x-.5,\y-.5)];
				
			\end{tikzpicture}} & \hspace{1cm} &

			\subcaptionbox{$(4,3,2)$-unimodal involution $\pi=476183259 \in \S_9$}{
			\begin{tikzpicture}[scale=.5]
				\newcommand\myx[1][(0,0)]{\pic at #1 {myx};}
				\tikzset{myx/.pic = {\draw 
 				     (-2.5mm,-2.5mm) -- (2.5mm,2.5mm)
 				     (-2.5mm,2.5mm) -- (2.5mm,-2.5mm);}}

				\draw[gray] (0,0) grid (9,9);
				\foreach \x/\y in {1/4,2/7,3/6,4/1,5/8,6/3,7/2,8/5,9/9}
				\myx[(\x-.5,\y-.5)];
				\draw[-, dashed] (0,0) -- (9,9);
				\draw[thick] (4,0) -- (4,9) (7,0) -- (7,9);
			\end{tikzpicture}}
		
		\end{tabular}
		\caption{ Pictorial representations of $\lambda$-unimodal permutations.}
		\label{fig:EXAMPLES}
	\end{figure} 
	
	Finally, we say a sequence of distinct positive integers $\gamma= \gamma_1\gamma_2\ldots\gamma_n$ is \textit{order-isomorphic} to a permutation $\pi \in \S_n$ if $\gamma_i<\gamma_j$ if and only if $\pi_i<\pi_j$. For example, the sequence $352$ is order-isomorphic to the permutation $231$ and the sequence $942185$ is order-isomorphic to the permutation $632154$.

\section{Character of the regular representation of $\S_n$}\label{sec: regular}
	
Let $\chi^R$ denote the character of the regular representation of $\S_n$, and let $\chi_\lambda^R$ denote the value that this representation takes on conjugacy class $\lambda$. The following theorem appears as Corollary 3.8 in \cite{AR2015}.
\begin{theorem}[{\cite[Cor. 3.8]{AR2015}}]\label{regular rep}
For $n\geq 1$ and $\lambda \vDash n$, 
$$\sum_{\pi\in\S^\lambda} (-1)^{\des_\lambda(\pi)} = \chi^R_\lambda.$$
\end{theorem}
In this section, we provide a combinatorial proof of this theorem using the following proposition concerning $\lambda$-unimodal permutations and $\lambda$-descents. If $\lambda =(\lambda_1, \lambda_2, \ldots, \lambda_k)$, let  ${n \choose \lambda}$ denote the multinomial coefficient given by
\[{n \choose \lambda}=\frac{n!}{\lambda_1!\lambda_2!\cdots\lambda_k!}.\]

	\begin{prop}\label{enumerate lambda descents}
The number of $\lambda$-unimodal permutations in $\S_n$ with $d$ $\lambda$-descents is
		\[{n \choose \lambda}{n-k \choose d},\] where $\lambda =(\lambda_1, \lambda_2, \ldots, \lambda_k)$.
\end{prop}
\begin{proof}
The multinomial coefficient ${n \choose \lambda}$ counts the partitions of $[n]$ into $k$ parts, the $i$-th of which is of size $\lambda_i$. The $i$-th segment of the $\lambda$-unimodal permutation is unimodal and comprised of the $\lambda_i$ elements determined by this partition. Since it is unimodal, it is enough to say which elements lie to the right of the maximum. Let $M\subseteq[n]$ be the set of $k$ elements that are the maximum in their part. Choose $d$ of the elements in $[n]\setminus M$ to lie to the right of the maximum in each part. There are ${n-k \choose d}$ ways for this to be done.
\end{proof}

For example, suppose that  $\lambda = (3,5,1)$ and the number of descents is $d=3$. If we take the partition of $[n]$ to be $\{\{1,6,8\},\{2,4,5,7,9\},\{3\}\}$, then $M=\{8,9,3\}$. If we choose our three elements from $[n]\setminus M$ to be $2,6,$ and 7, then we obtain the permutation $186459723$, which is a $(3,5,1)$-unimodal permutation with three $\lambda$-descents.

The next corollary follows immediately from Proposition \ref{enumerate lambda descents}. 

\begin{cor}\label{enumerate lambda}
For $\lambda=(\lambda_1,\lambda_2, \ldots, \lambda_k) \vDash n$, the number of $\lambda$-unimodal permutations in $\S^\lambda$ is 
	\[{n \choose \lambda}2^{n-k}.\]
	\end{cor}
	
	In our proof of Theorem \ref{regular rep}, we show that the number of $\lambda$-unimodal permutations with an even number of $\lambda$-descents minus the number of $\lambda$-unimodal permutations with an odd number of $\lambda$-descents coincides with the character of the regular representation on conjugacy class $\lambda$. 
	
	\begin{proof}[Proof of Theorem \ref{regular rep}]
	First, notice that when $\lambda=(1,1,\ldots,1)$, there are exactly zero $\lambda$-descents. Also, every permutation is trivially a $\lambda$-unimodal permutation, so in the case where $\lambda = (1,1,\ldots,1)$, we have
	$$\sum_{\pi\in\S^\lambda} (-1)^{\des_\lambda(\pi)} = \sum_{\pi\in\S_n} 1=n!.$$
	
	Now if $\lambda\neq(1,1,\ldots,1)$, then we must have $n-k>0$ and thus,
	$$\sum_{\pi\in\S^\lambda} (-1)^{\des_\lambda(\pi)} ={n \choose \lambda} \sum_{d=0}^{n-k} (-1)^d{n-k \choose d} = {n \choose \lambda} (1-1)^{n-k}=0. $$
	Thus the alternating sum is $n!$ when $\lambda = (1, 1, \ldots, 1)$ and is 0 otherwise, which exactly coincides with the character of the regular representation.
	\end{proof}

\section{$\lambda$-unimodal involutions}\label{sec3}

 In this section, we let $\Lambda_k$ denote the set of integer compositions into $k$ positive integer parts and let $\lambda=(\lambda_1, \lambda_2, \ldots, \lambda_k)\in \Lambda_k$. Let $\x$ denote the set of indeterminates $\{x_1, x_2, \ldots, x_k\}$. If $F$ is a generating function on variables $\{x_{i_1}, x_{i_2}, \ldots, x_{i_j}\}$, we write $F(\x_V)$ where $V=\{i_1,i_2, \ldots, i_j\}\subseteq\{1,2,\ldots,k\}$. It is occasionally useful to use alternative notation. If $F$ is a generating function on variables $\x \setminus\{x_{i_1}, x_{i_2}, \ldots, x_{i_j}\}$, we also write $F(\x; \hat{x}_{i_1}, \hat{x}_{i_2}, \ldots, \hat{x}_{i_j})$. 
For example, if $k=3$, then $\x=\{x_1, x_2,x_3\}$ and $F(\x_{\{1,3\}})$, or $F(\x; \hat{x}_2)$, is a function on variables $x_1$ and $x_3$ only. 

Let $x^\lambda$ denote the monomial $x^\lambda=x_1^{\lambda_1}x_2^{\lambda_2}\cdots x_k^{\lambda_k}$. 
Define the generating function for $\lambda$-unimodal involutions with $\lambda \in \Lambda_k$ as 
$$L^k(\x) = \sum_{\lambda \in \Lambda_k}|\Inv^\lambda| x^\lambda.$$
Then we have the following theorem.
\begin{theorem}\label{main theorem}
We have $L^0(\x)=1$, $L^1(\x) = \dfrac{x_1}{(1-x_1)^2}$ and for $k \geq 2$,
\begin{flalign*}
L^k(\x)&= \hspace{-.45cm}\quad\frac{(1+x_1^2)\sum_{i=0}^{k-1} 2^{k-i-1} \sum_{V}L^i(\x_V)\prod_{j\not\in V} x_1x_j\prod_{j\in V}(1+x_1x_j)^2}{2(1-x_1)^2\prod_{\ell=2}^k(1-x_1x_\ell)^2} - \frac{L^{k-1}(\x; \hat{x}_1)}{2}
\end{flalign*}
where the second summation occurs over all subsets $V$ of $\{2, 3, \ldots, k\}$ of size $i$.
\end{theorem}

To prove Theorem~\ref{main theorem}, we start with a lemma establishing the initial condition and an observation about unimodal involutions. For $\pi \in \S_n$ with $\pi=\pi_1\pi_2\ldots\pi_n$ and $\sigma \in \S_m$ with $\sigma=\sigma_1\sigma_2\ldots\sigma_m$, define $\pi \oplus \sigma \in \S_{n+m}$ by
	\[\pi \oplus \sigma = \pi_1\pi_2\ldots\pi_n(\sigma_1+n)(\sigma_2+n)\ldots(\sigma_m+n).\]
For example, if $\pi=312 \in \S_3$ and $\sigma=635421\in \S_6$, then $\pi \oplus \sigma=312968754 \in \S_9$.

\begin{lemma}\label{lem:L1}
If $n\geq 1$, then there are $n$ unimodal involutions in $\S_n$. Consequently, the generating function $L^1(x) =\displaystyle \sum_{n=1}^\infty\sum_{\pi\in\Inv^{(n)}}\hspace{-5pt} x^n$ is given by $$L^1(x) = \frac{x}{(1-x)^2}.$$ Furthermore, 
\begin{itemize}
\item these are exactly the permutations $\iota_j \oplus \delta_{n-j}$, where $\iota_j$ is the increasing (identity) permutation of length $j$ for $0\leq j \leq n-1$ and $\delta_{n-j}$ is the decreasing permutation of length $n-j$; and
\item these are exactly the unimodal permutations in $\S_n$ whose inverse is also unimodal.
\end{itemize}
\end{lemma}

\begin{proof}
Clearly there is only one unimodal permutation of length 1 and it is an involution. We proceed by induction. For any $\pi\in \Inv^{(n)}$ with $n\geq 2$, either $\pi_1=1$ or $\pi_n=1$. Notice that if $\pi_1=1$, then $\pi\in \Inv^{(n)}$ if and only if $\pi=1\oplus \sigma$ with $\sigma \in \Inv^{(n-1)}$. If $\pi_n=1$, then necessarily $\pi_1=n$, and thus $\pi$ is the decreasing permutation that is indeed an involution. Therefore, $|\Inv^{(n)}|=|\Inv^{(n-1)}|+1$, which in turn implies that $|\Inv^{(n)}|=n$, and thus the enumerative result follows. Clearly, these $n$ distinct involutions are the permutations of the form $\iota_j \oplus \delta_{n-j}$ for $0\leq j\leq n-1$.

The second observation follows from Proposition $16^*$ in \cite{SS1985} by noting that a permutation in $\pi \in \S_n$ is unimodal with a unimodal inverse exactly when it avoids the classical patterns $213, 312,$ and $231$. However, we can also directly prove this fact. Let $\pi\in\S_n$ be unimodal with $\pi_k=n$. If it's inverse is also unimodal, it must be the case that the subsequence $\alpha$ of $\pi$ given by taking the elements less than $\pi_n$ is an increasing sequence (because the graph of $\pi^{-1}$ can be obtained by reflecting the graph of $\pi$ about the diagonal). Similarly, the subsequence $\beta$ of $\pi$ obtained by taking elements greater than or equal to $\pi_n$ is a decreasing sequence. Since $\pi$ is unimodal we know that the sequence $\pi_k\ldots \pi_n$ is decreasing and in particular for any $k\leq i \leq n$, we have $\pi_i\geq\pi_n$. Therefore the increasing sequence $\alpha$ contains no elements from $\pi_k\ldots \pi_n$. Therefore, $\beta$ must include all of the sequence $\pi_k\ldots \pi_n$, and cannot include element $\pi_i$ with $i<k$ since $\pi_k=n$. This implies that the number of unimodal permutations that have a unimodal inverse is $n$, and so they must be those described in this lemma. 
\end{proof}

In the remainder of this section, for a permutation $\pi \in \Inv^\lambda$ with $\lambda\in\Lambda_k$, let us write $\pi=\beta_1\beta_2\ldots\beta_k$ where $|\beta_i|=\lambda_i$. Then $\beta_i$ is unimodal for each $i$. Furthermore, we can write $\beta_1$ as $\beta_1=\gamma_1\gamma_2\ldots\gamma_k\eta_k\ldots\eta_2\eta_1$ where $\gamma_i$ is an increasing sequence, $\eta_i$ is a decreasing sequence, and elements of $\gamma_i$ and $\eta_i$ lie between $\lambda_1+\cdots + \lambda_{i-1}+1$ and $\lambda_1+\cdots + \lambda_{i-1} +\lambda_i.$  We denote $\alpha_i:=\gamma_i\eta_i$.
For example, consider the permutation $\pi = 476183259$ as a $(4,3,2)$-unimodal involution (pictured in Figure~\ref{fig:EXAMPLES}(b)). Then $\beta_1 = 4761$, $\beta_2=832$, and $\beta_3=59$. In this case, $\alpha_1 = 41$, $\alpha_2=76$, and $\alpha_3$ is empty. As another example, consider $\pi = 127865349$ as a $(6, 3)$-unimodal involution. In this case, $\beta_1 = 127853$, $\beta_2 = 349$, $\alpha_1 = 1265$, and $\alpha_2 = 78$. 

We also denote by $\bar{\alpha_i}$ the subsequence of $\beta_i$ consisting of elements that are less than or equal to $\lambda_1$. Since $\beta_i$ is unimodal for all $i$, $\bar{\alpha_i}$ must be comprised of (at most) two contiguous segments, one increasing and one decreasing. Since $\pi$ is an involution, $|\alpha_i|=|\bar{\alpha}_i|$ and moreover, $\alpha_i$ and $\bar{\alpha_i}$ are reflections of each other about the diagonal. Since they are both unimodal, by Lemma~\ref{lem:L1}, the are both order-isomorphic to the same unimodal involution. This will be important in the proof of Theorem~\ref{main theorem}.

To give some indication of how the proof will proceed, let us see how this recurrence works for the case $k=2$. For ease of notation, we  set $x:=x_1$ and $y:=x_2$. In this case, the theorem states that 
$$ L^2(x,y) =\frac{(1+x^2)}{2(1-x)^2(1-xy)^2}[2xyL^0 + (1+xy)^2L^1(y)] - \frac{L^1(y)}{2}.$$ 
We can rewrite this as:
$$ L^2(x,y) =\frac{xy(1+x^2)}{(1-x)^2(1-xy)^2}L^0 + \frac{x}{(1-x)^2}L^1(y) + \frac{(1+x^2)}{(1-x)^2}\cdot\frac{1}{2}\left[\frac{(1+xy)^2}{(1-xy)^2}-1\right]L^1(y).$$ 
Note that there are essentially three terms in this sum. These three terms correspond to the following cases, also illustrated in Figure~\ref{fig:L2}.
\begin{enumerate}[(i)]
\item The first term, $\frac{xy(1+x^2)}{(1-x)^2(1-xy)^2}L^0$ corresponds to the case when $\pi_j\leq \lambda_1$ for all $j>\lambda_1$.
\item The second term, $ \frac{x}{(1-x)^2}L^1(y)$ corresponds to the case when $\pi_j\leq \lambda_1$ for all $j\leq \lambda_1$. 
\item The third term, $ \frac{(1+x^2)}{(1-x)^2}\cdot\frac{1}{2}\left[\frac{(1+xy)^2}{(1-xy)^2}-1\right]L^1(y)$ corresponds to other cases. 
\end{enumerate}

%

	\begin{figure}[!h]
	\captionsetup{width=0.88\textwidth}
	\centering
	\resizebox{.88\textwidth}{!}{%
	\begin{tabular}{@{}l@{}}
	\begin{tikzpicture}[scale=.7]
			\newcommand\myx[1][(0,0)]{\pic at #1 {myx};}
			\tikzset{myx/.pic = {\draw 
			      (-2mm,-2mm) -- (2mm,2mm)
			      (-2mm,2mm) -- (2mm,-2mm);}}
				\fill[red!10!white] (6,6) rectangle (9,9);
				\draw[black] (0,0) rectangle (9,9);
				\draw[black] (0,0) rectangle (6,9);
				\draw[black] (0,0) rectangle (9,6);
				\draw[gray] (2,0) rectangle (4,9);
				\draw[gray] (0,2) rectangle (9,4);
				\draw[b] (0,0) -- (6,0) node[l] {\footnotesize $\lambda_1$};
				\draw[b] (6,0) -- (9,0) node[l] {\footnotesize $\lambda_2$};
				\draw [dotted, thick] (.1,.1)--(1.9,1.9){};
				\draw [dotted, thick] (4.1,5.9)--(5.9,4.1){};
				\draw [dotted, thick] (2.1,6.1)--(2.9,7.4){};
				\draw [dotted, thick] (3.0,8.9)--(3.9,7.5){};
				\draw [dotted, thick] (6.1,2.1)--(7.4,2.9){};
				\draw [dotted, thick] (8.9,3)--(7.5,3.9){};
\draw[black] (0,0) rectangle (9,9);
		\end{tikzpicture}
\hspace{12pt}
	\begin{tikzpicture}[scale=.7]
			\newcommand\myx[1][(0,0)]{\pic at #1 {myx};}
			\tikzset{myx/.pic = {\draw 
			      (-2mm,-2mm) -- (2mm,2mm)
			      (-2mm,2mm) -- (2mm,-2mm);}}
				\fill[red!10!white] (0,3) rectangle (3,9);
				\fill[red!10!white] (3,0) rectangle (9,3);
				\draw[black] (0,0) rectangle (9,9);
				\draw[black] (0,0) rectangle (3,9);
				\draw[black] (0,0) rectangle (9,3);
				\draw[b] (0,0) -- (3,0) node[l] {\footnotesize $\lambda_1$};
				\draw[b] (3,0) -- (9,0) node[l] {\footnotesize $\lambda_2$};
				\node at (5.5,5.5) {$L^1(y)$};
				\node at (1.5,1.5) {$L^1(x)$};
\draw[black] (0,0) rectangle (9,9);
		\end{tikzpicture}\hspace{12pt}
		\begin{tikzpicture}[scale=.7]
			\newcommand\myx[1][(0,0)]{\pic at #1 {myx};}
			\tikzset{myx/.pic = {\draw 
			      (-2mm,-2mm) -- (2mm,2mm)
			      (-2mm,2mm) -- (2mm,-2mm);}}
				\fill[red!10!white] (4,4) rectangle (9,9);
				\fill[white] (5,5) rectangle (7,7);
				\fill[red!10!white] (0,5) rectangle (5,7);
				\fill[red!10!white] (5,7) rectangle (7,9);
				\fill[red!10!white] (7,5) rectangle (9,7);
				\fill[red!10!white] (5,0) rectangle (7,5);
				\draw[black] (0,0) rectangle (9,9);
				\draw[black] (0,0) rectangle (4,9);
				\draw[black] (0,0) rectangle (9,4);
				\draw[gray] (0,5) rectangle (9,7);
				\draw[gray] (5,0) rectangle (7,9);
				\draw[gray] (1,0) rectangle (3,9);
				\draw[gray] (0,1) rectangle (9,3);
				\draw[b] (0,0) -- (4,0) node[l] {\footnotesize $\lambda_1$};
				\draw[b] (4,0) -- (9,0) node[l] {\footnotesize $\lambda_2$};
				\node at (6,6) {$L^1(y)$};
				\draw [dotted, thick] (0.1,0.1)--(.9,.9){};
				\draw [dotted, thick] (1.1,4.1)--(1.7,4.9){};
				\draw [dotted, thick] (1.8,8.9)--(2.9,7.1){};
				\draw [dotted, thick] (4.1,1.1)--(4.9,1.7){};
				\draw [dotted, thick] (8.9,1.8)--(7.1,2.9){};
				\draw [dotted, thick] (3.1, 3.9)--(3.9,3.1){};
\draw[black] (0,0) rectangle (9,9);
		\end{tikzpicture}
				 \end{tabular}%
		 }
		\caption{This figure illustrates the three possibilities for when $k=2$. The shaded-in regions are empty and the dotted lines indicate increasing and decreasing regions. These correspond to cases (i), (ii), and (iii), respectively in the example above.}
		\label{fig:L2}
	\end{figure}

Let us start with case (i). In this case, there are no elements in the top right quadrant of size $(\lambda_2\times\lambda_2)$ of the graph of $\pi$. Therefore each element $\pi_j$ with $j>\lambda_1$ must be less than or equal to $\lambda_1$ in magnitude. Since $\pi$ is an involution and thus symmetric about the diagonal, for each such element $\pi_j$, there will be a corresponding element $j=\pi_{\pi_j}$ in the upper left region of the graph of $\pi$. Therefore, for each $y$ contributed to the generating function (corresponding to an element in $\beta_2$), there is also an $x$ contributed (corresponding to an element in $\beta_1$). Additionally, since the elements $\pi_{\lambda_1+1}\ldots\pi_{\lambda_2}$ must be a unimodal sequence whose reflection about the diagonal is also unimodal, this sequence must be order-isomorphic to a unimodal involution $\alpha_2$ by Lemma~\ref{lem:L1}. The sequence of elements in the lower left $(\lambda_1\times\lambda_1)$ quadrant must be order-ismorphic to a (possibly empty) unimodal involution $\alpha_1$.
Additionally, since the reflection of $\beta_2$ (order-isomorphic to $\alpha_2$) consists of the largest entries in the unimodal sequence $\beta_1$, these entries must be contiguous. They appear either immediately before or immediately after the peak of $\alpha_1$; see Figure~\ref{fig:L2}. 

There are $\lambda_2$ options for the (nonempty) involution $\alpha_2$, and (if $\lambda_1>\lambda_2$), $\lambda_1-\lambda_2$ options for the involution $\alpha_1$. If $\alpha_1$ is nonempty, there are two options for how $\alpha_1$ and $\alpha_2$ fit together in $\beta_1$. 
Therefore, we obtain the term:  
 $$(1+2(x+2x^2+3x^3+4x^4+ \cdots))(xy+2x^2y^2+ 3x^3y^3+4x^4y^4+\cdots) = \frac{(1+x^2)}{(1-x)^2}\cdot\frac{xy}{(1-xy)^2}.$$

In case (ii), we have the direct sum of two involutions and so the generating function must be $L^1(x)\cdot L^1(y)$. Now we only need to consider case (iii).

In this last case, there are some elements $\pi_j>\lambda_1$ with $j>\lambda_1$. In the graph of the permutation, these are in the upper right $(\lambda_2\times\lambda_2)$ quadrant. Since $\beta_2$ is unimodal, the elements in the upper right quadrant must form a contiguous unimodal sequence. Because  $\pi$ is symmetric about the diagonal (being an involution), this sequence must itself be order-isomorphic to a unimodal involution, with generating function given by $L^1(y)$.

The sequence of elements in $\beta_2$ that are less than or equal to $\lambda_1$, denoted $\bar\alpha_2$ (which is nonempty in this case), is unimodal and  the elements of $\beta_2$ that take values greater than $\lambda_1$ must appear immediately to the left or to the right of the peak of $\bar\alpha_2$. This is necessary to guarantee that $\beta_2$ is unimodal. In addition, since the elements of $\bar\alpha_2$ must be less than or equal to $\lambda_1$ in magnitude, and since $\pi$ is an involution and thus symmetric about the diagonal,  there will be corresponding elements in the upper left region of the graph of $\pi$. Thus, $\bar\alpha_2$ is order-isomorphic to a  unimodal involution (by Lemma~\ref{lem:L1}). For each nonempty choice of $\bar\alpha_2$, there are two possibilities, before or after the peak, for where the elements larger than $\lambda_1$ in magnitude will lie. 

Notice that we could have $\alpha_1$ empty or not in this case. If it is nonempty, for every choice of $\alpha_1$, there are two possible places where $\alpha_2$ could go, before or after the peak of $\alpha_1$. Given a choice for $\alpha_2$, there are two more choices for where the elements greater than $\lambda_1$ in magnitude could go in $\beta_2$. 
Taken altogether, we get
$$(1+2(x+2x^2+3x^3+ \cdots))(2xy+4x^2y^2+ 6x^3y^3+\cdots)L^1(y) = \frac{(1+x^2)}{(1-x)^2}\cdot\frac{2xy}{(1-xy)^2}L^1(y),$$
which is equivalent to the generating function written for case (iii).

The main theorem follows by similar reasoning. We proceed by establishing some lemmas. 
Let us start with the term involving $L^0$, corresponding to the case when for all $i>\lambda_1$, we have $\pi_i\leq \lambda_1$. This case will inform what happens for all other terms as well.

\begin{lemma}\label{lem:L0case}
Let $k\geq 2$. 
The generating function for elements of $\Inv^\lambda$ such that  for all $i>\lambda_1$ we have $\pi_i\leq \lambda_1$ is given by
$$
2^{k-2}\cdot\frac{(1+x_1^2)}{(1-x_1)^2}\cdot\prod_{i=2}^k \frac{x_1x_i}{(1-x_1x_i)^2}.
$$
\end{lemma}
\begin{proof}
Let us first notice that the generating function given in the statement of the lemma is actually the product: 
$$
\bigg(1+2\sum_{n\geq1} nx_1^n\bigg)\cdot 2\bigg(\sum_{n\geq1} nx_1^nx_2^n\bigg) \cdot 2\bigg(\sum_{n\geq1} nx_1^nx_3^n\bigg)\cdot\cdots\cdot 2\bigg(\sum_{n\geq1} nx_1^nx_{k-1}^n\bigg)\cdot\bigg(\sum_{n\geq1} nx_1^nx_k^n\bigg).
$$
In this case, since for all $i>\lambda_1$ we have $\pi_i\leq \lambda_1$, the element $\pi_i$ corresponds to an element $i=\pi_{\pi_i}$ (by reflection about the diagonal) that lies in $\beta_1$, i.e., the first $\lambda_1$ elements of $\pi$. It is therefore enough to determine $\beta_1$. 
As before, we let $\alpha_1$ be the subsequence of $\beta_1$ consisting of the elements that are less than or equal to $\lambda_1$, let $\alpha_2$ be the subsequence of $\beta_1$ consisting of the elements that are between $\lambda_1+1$ and $\lambda_1+\lambda_2$, etc. 

In these cases, $\alpha_i$ is nonempty for each $i>1$ (since $\lambda_i$ is strictly positive) but $\alpha_1$ may be empty. By Lemma~\ref{lem:L1}, for each $i$, $\alpha_i$ is order-isomorphic to a unimodal involution, since it is unimodal and its reflection about the diagonal is also unimodal.  Therefore, for each $i>1$, the number of ways to pick $\alpha_i$ (and equivalently $\beta_i$) is given by $L^1(x_1x_i)$. The number of ways to pick $\alpha_1$ is $1+L^1(x_1)$ since $\alpha_1$ could potentially be empty.

Finally, notice that for each $i$, $\alpha_i$ is of the form $\alpha_i=\gamma_i\eta_i$ where $\gamma_i$ is a contiguous increasing segment of $\pi$ and $\eta_i$ is a contiguous decreasing segment of $\pi$, and so either the last element of $\gamma_i$ or the first element of $\eta_i$ is the peak of $\alpha_i$. In order for $\beta_1$ to be unimodal, for each $i<k$ with $\alpha_i$ nonempty there are exactly two options to split $\alpha_i$ into $\gamma_i$ and $\eta_i$: before or after the peak of $\alpha_i$. Thus the result follows.
\end{proof}

In the next lemma, we generalize the previous lemma and consider the cases where there exists some nonempty collection of $\beta_i$'s whose values all lie below $\lambda_1$. 

\begin{lemma}\label{lem:LV}
Let $k\geq 2$ and $V\subsetneq\{2, 3, \ldots, k\}$ with $|V|=i$. Let $\Inv_V^\lambda$ be the set of permutations $\pi \in \Inv^\lambda$, with $\pi=\beta_1\beta_2\ldots\beta_k$ where $|\beta_j|=\lambda_j$, so that $j\not\in V$ if and only if every element of $\beta_j$ is less than or equal to $\lambda_1$. Then the generating function for $\Inv_V^\lambda$ is given by 
$$ 2^{k-i-2} L^i(\x_V)\frac{(1+x_1^2)}{(1-x_1)^2}\prod_{j\not\in V} \frac{x_1x_j}{(1-x_1x_j)^2}\prod_{j\in V}\frac{(1+x_1x_j)^2}{(1-x_1x_j)^2}.$$
\end{lemma}

\begin{proof}
First notice that for each $j\in V$, the term given in the product is $$\frac{(1+x_1x_j)^2}{(1-x_1x_j)^2}=1 + 4\cdot \sum_{n\geq 1} n x_1^nx_j^n$$ and so the generating function given in the statement of the lemma is actually: 
$$ 2^{k-i-2} \bigg(1+2\sum_{n\geq1} nx_1^n\bigg)\cdot\prod_{j\not\in V} \bigg(\sum_{n\geq1} nx_1^nx_j^n\bigg)\cdot\prod_{j\in V}\bigg(1 + 4\cdot \sum_{n\geq 1} n x_1^nx_j^n\bigg) \cdot L^i(\x_V).$$
To determine the elements of $\beta_1$, the proof is similar to that of Lemma~\ref{lem:L0case}, except that in the cases where $j\in V$, $\alpha_j$ can be empty. In other words, $\beta_j$ does not necessarily have any elements that are less than or equal to $\lambda_1$ in value. If $\beta_j$ does have elements that are less than or equal to $\lambda_1$ in value, denoted $\bar{\alpha}_j$, then there are also two possibilities for where the contiguous unimodal segment of values greater than $\lambda_1$ can appear in the segment $\beta_j$: immediately before or immediately after the peak of $\bar{\alpha}_j$. Thus for each $j<k$ with $j\in V$ and $\bar{\alpha}_j$ nonempty, there are 4 options instead of just 2.

Finally, the elements that appear in the upper right $(\lambda_2+\cdots+\lambda_k)\times(\lambda_2+\cdots+\lambda_k)$ region of the graph of $\pi$ must be order-isomorphic to a $\nu$-unimodal involution where $\nu$ is some composition of length $i$ and so the generating function for that region is given by  $L^i(\x_V)$.
\end{proof}

Finally, let us deal with the case when for all $2\leq i\leq k$, each $\beta_i$ has an element greater than $\lambda_1$. The reason this case is slightly more complicated is that, in this case, we can have $\pi_i<\lambda_1$ for all $i\leq \lambda_1$ (equivalently, $\pi_i>\lambda_1$ for all $i>\lambda_1$).

\begin{lemma}\label{lem:Lk-1case}
Let $k\geq 2$ and let $\hat\Inv^\lambda$ be the set of permutations $\pi \in \Inv^\lambda$, with $\pi=\beta_1\beta_2\ldots\beta_k$ where $|\beta_i|=\lambda_i$, so that there is an element of $\beta_i$ that is greater than $\lambda_1$ for all $2\leq i\leq k$. Then the generating function for $\hat\Inv^\lambda$ is 
$$\frac{1}{2}\bigg[\frac{(1+x_1^2)}{(1-x_1)^2} \cdot \prod_{j=2}^k \frac{(1+x_1x_j)^2}{(1-x_1x_j)^2} - 1\bigg]L^{k-1}(\x; \hat{x}_1).$$
\end{lemma}
\begin{proof}
Let us first notice that this generating function can be written as the sum of the following summands: 
$$
\frac{x_1}{(1-x_1)^2} L^{k-1}(\x;\hat{x}_1)
$$
and 
$$
\frac{(1+x_1^2)}{(1-x_1)^2} \cdot \frac{1}{2}\bigg[\prod_{j=2}^k \frac{(1+x_1x_j)^2}{(1-x_1x_j)^2} - 1\bigg]L^{k-1}(\x; \hat{x}_1).
$$
In the first summand, we count permutations $\pi$ so that for all $2\leq j\leq k$, we have $\alpha_j$ empty, and in the second summand, we count the permutations $\pi$ so that there is some $2\leq j\leq k$ with $\alpha_j$ nonempty. 
In the first case, we have the direct sum of a unimodal involution with a $\nu$-unimodal permutation where $\nu$ is a composition of length $k-1$ and so the generating function is given by the first summand above. 

In the second case, we have 
$$ \bigg(1+2\sum_{n\geq1} nx_1^n\bigg)\cdot \frac{1}{2}\bigg[\prod_{j=2}^k\bigg(1+4\cdot\sum_{n\geq 1} n x_1^nx_j^n\bigg) - 1\bigg]\cdot L^{k-1}(\x; \hat{x}_1).$$ The proof is similar to that of Lemma~\ref{lem:LV}. Here, $\alpha_1$ could be empty (accounting for the 1 in the first factor), or not. Since there is some $\alpha_j$ with $j\geq 2$ that is nonempty, we again have two options for each unimodal involution isomorphic to $\alpha_1$, as in the proof of Lemma~\ref{lem:L0case}.
For each non-empty segment $\bar{\alpha}_j$ with $j\geq 2$, the segment is order-isomorphic to a unimodal involution and there are (usually) four possibilities for where $\bar\alpha_j$ may be split for the inclusion of elements greater than $\lambda_1$ and for where $\alpha_j$ may be split (in $\beta_1$). If $j$ is the largest index with $\alpha_j$ nonempty, then there are only two possibilities overall. Multiplying by 1/2 accounts for this. Finally, we must subtract the possibility that $\alpha_j$ is empty for all $j\geq 2$.
\end{proof}
\begin{proof}[Proof of Theorem~\ref{main theorem}]
The proof of the main theorem is obtained by adding the results from Lemmas \ref{lem:LV} and \ref{lem:Lk-1case} combined with the initial condition given in Lemma~\ref{lem:L1}. 
\end{proof}

\section{$\lambda$-descents and involutions}

We can refine the previous proof, using the indeterminate $t$ to keep track of descents. 
Let $\Inv^\lambda(d)$ denote the set of involutions $\pi\in\Inv^\lambda$ such that $\des_\lambda(\pi) = d$.  Define
$$
L^k(\x,t) = \sum_{\lambda \in \Lambda_k}\sum_{d\geq0}|\Inv^\lambda(d)| x^\lambda t^{d}. $$
 This leads to a refined version of Theorem~\ref{main theorem}. 

\begin{theorem}\label{thm descents}
We have $L^0(\x,t)=1$, $L^1(\x,t) = \dfrac{x_1}{(1-x_1)(1-tx_1)}$ and for $k \geq 2$,
\begin{flalign*}
L^k(\x,t)=&\quad\frac{(1+tx_1^2)\sum_{i=0}^{k-1} (1+t)^{k-i-1} \sum_{V}L^i(\x_V,t)\prod_{j\not\in V} x_1x_j\prod_{j\in V}(1+tx_1x_j)^2}{(1+t)(1-x_1)(1-tx_1)\prod_{\ell=2}^k(1-x_1x_\ell)(1-t^2x_1x_\ell)}\\
& -\frac{L^{k-1}(\x,t;\hat{x}_1)}{1+t}.
\end{flalign*}
where the sum is over all subsets $V$ of $\{2, 3, \ldots, k\}$ with size $i$.
\end{theorem}

Again, let us first show that the initial conditions hold.

\begin{lemma}\label{lem:L1 descents}
If $n\geq 1$, then there are $n$ unimodal involutions in $\S_n$. Consequently, the generating function $L^1(x,t) =\displaystyle \sum_{\pi\in\Inv^{(n)}}\hspace{-5pt} x^nt^{\des(\pi)}$ is given by $$L^1(x,t) = \frac{x}{(1-x)(1-xt)}.$$
\end{lemma}

\begin{proof}
The $n$ unimodal involutions on $[n]$ are exactly of the form  $\iota_j \oplus \delta_{n-j}$  for $0\leq j \leq n-1$ where $\iota_j$ is the increasing permutation of length $j$ and $\delta_{n-j}$ is the decreasing permutation of length $n-j$. Therefore, for each $n \geq 1$ and each $0\leq i \leq n-1$ there is exactly one unimodal involution of length $n$ with $i$ descents. The generating function is given by $$\sum_{n\geq1} (1+t+t^2+\cdots + t^{n-1}) x^{n}$$ which is equivalent to the one given in the statement of the theorem. 
\end{proof}

Let us now provide an analog of Lemma~\ref{lem:LV}. 

\begin{lemma}\label{lem:LV-des}
Let $k\geq 2$ and $V\subsetneq\{2, 3, \ldots, k\}$ with $|V|=i$. Let $\Inv_V^\lambda$ be the set of permutations $\pi \in \Inv^\lambda$, with $\pi=\beta_1\beta_2\ldots\beta_k$ where $|\beta_j|=\lambda_j$, so that $j\not\in V$ if and only if every element of $\beta_j$ is less than or equal to $\lambda_1$. Then the generating function for $\Inv_V^\lambda$ is given by 
$$ (1+t)^{k-i-2} L^i(\x_V,t)\frac{(1+tx_1^2)}{(1-x_1)(1-tx_1)}\prod_{j\not\in V} \frac{x_1x_j}{(1-x_1x_j)(1-t^2x_1x_j)}\prod_{j\in V}\frac{(1+tx_1x_j)^2}{(1-x_1x_j)(1-t^2x_1x_2)}.$$
\end{lemma}

\begin{proof}
We follow the proof of Lemma~\ref{lem:LV}, keeping track of descents along the way. The formula in Lemma~\ref{lem:LV-des} is equivalent to 
\begin{align*}
(1+t)^{k-i-2} \cdot L^i(\x_V,t)&\cdot\bigg(1+(1+t)\sum_{n\geq1} (1+t+t^2+\cdots +t^{n-1})x_1^n\bigg) \\
&\cdot\prod_{j\not\in V} \bigg(\sum_{n\geq1} (1+t^2+t^4+\cdots +t^{2(n-1)})x_1^nx_j^n\bigg) \\
&\cdot\prod_{j\in V}\bigg(1 + (1+t)^2\cdot \sum_{n\geq 1} (1+t^2+t^4+\cdots +t^{2(n-1)}) x_1^nx_j^n\bigg).
\end{align*}
As before, we build a $\lambda$-unimodal involution by constructing $\beta_1$ as a nested sequence of unimodal involutions, $\alpha_i$ for $1\leq i\leq k$. The rest of the involution is given by reflection about the diagonal and $L^i(\x_V,t)$. 

In  this formula, the sum $$\bigg(1+(1+t)\sum_{n\geq1} (1+t+t^2+\cdots +t^{n-1})x_1^n\bigg)$$ corresponds to the segment $\alpha_1$, i.e., those elements $\pi_j$ with $j\leq \lambda_1$ and $\pi_j\leq \lambda_1$. If $\alpha_1$ is nonempty, we multiply by $(1+t)$ to account for whether the remaining elements in $\beta_1$ will occur before or after the peak of $\alpha_1$. If they occur before the peak, a descent is added and if they occur after,  no descents are added. 

For each $j\geq 2$, one also obtains $(1+t^2+t^4+\cdots +t^{2(n-1)})$ as the coefficient of $x_1^nx_j^n$ since any descent of $\alpha_j$ occurs again for $\bar\alpha_j$ (since $\alpha_j$ and $\bar\alpha_j$ are order-isomorphic to the same involution).  For $j\in V$, there is an extra copy of $(1+t)$ to account for where the additional elements of $\beta_j$ go, before or after the peak of $\bar{\alpha}_j$. 
\end{proof}

Finally, we provide an analog of Lemma~\ref{lem:Lk-1case}.

\begin{lemma}\label{lem:Lk-1case-des}
Let $k\geq 2$ and let $\hat\Inv^\lambda$ be the set of permutations $\pi \in \Inv^\lambda$, with $\pi=\beta_1\beta_2\ldots\beta_k$ where $|\beta_i|=\lambda_i$, so that there is an element of $\beta_i$ that is greater than $\lambda_1$ for all $i$. Then the generating function for $\hat\Inv^\lambda$ is 
$$\frac{1}{1+t}\cdot \bigg[\frac{(1+tx_1^2)}{(1-x_1)(1-tx_1)}\cdot\prod_{j=2}^k\frac{(1+tx_1x_j)^2}{(1-x_1x_j)(1-t^2x_1x_j)} - 1\bigg]\cdot L^{k-1}(\x,t; \hat{x}_1).$$
\end{lemma}
\begin{proof}
Let us first notice that this generating function can be written as the sum of the following summands: 
$$
\frac{x_1}{(1-x_1)(1-tx_1)}L^{k-1}(\x,t;\hat{x}_1)
$$
and 
$$
\frac{(1+tx_1^2)}{(1-x_1)(1-tx_1)}\cdot\frac{1}{1+t}\cdot \bigg(\prod_{j=2}^k\frac{(1+tx_1x_j)^2}{(1-x_1x_j)(1-t^2x_1x_j)} - 1\bigg)\cdot L^{k-1}(\x,t; \hat{x}_1).
$$
As in Lemma~\ref{lem:Lk-1case}, in the first summand, we count permutations $\pi$ so that for all $i\leq \lambda_i$, $\pi_i\leq \lambda_1$ and in the second summand, we count those permutations so that there is some $i\leq \lambda_1$ with $\pi_i>\lambda_1$. These formulas follow from the same reasoning as in the proof of Lemma~\ref{lem:LV-des}. 
\end{proof}

\begin{proof}[Proof of Theorem~\ref{thm descents}]
The proof follows directly from Lemmas~\ref{lem:L1 descents}, \ref{lem:LV-des}, and \ref{lem:Lk-1case-des}. 
\end{proof}

We end with an observation about the function $$g(\x,t):=\prod_{j=2}^k(1+tx_1x_j)^2 - \prod_{j=2}^k(1-x_1x_j)(1-t^2x_1x_j),$$ which can be proven either by induction or by  noticing that $g(\x,-1) = g'(\x,-1)=0$ (where the derivative is with respect to the variable $t$). 
\begin{lemma}\label{lem:g}
For $k\geq 2$,  $$\frac{g(\x,t)}{(1+t)^2} = \frac{1}{(1+t)^2}\left[\prod_{j=2}^k(1+tx_1x_j)^2 - \prod_{j=2}^k(1-x_1x_j)(1-t^2x_1x_j)\right]$$ is a polynomial in $\x$ and $t$.
\end{lemma}

This lemma will allow us to evaluate the generating function $L^k(\x, t)$ at $t=-1$ in the next section in order to obtain a generating function for the Gelfand character. 

\section{The Gelfand character}

In this section, we state the main theorem of the paper. Define
$$G^k(\x) = \sum_{\lambda\in \Lambda_k} \chi_\lambda^G x^\lambda.$$ where $\chi^G$ is the Gelfand character mentioned in the introduction, $\Lambda_k$ is the set of integer compositions of length $k$, and $\chi_\lambda^G$ is the value the character takes on the conjugacy class given by $\lambda$. 

Notice that by Equation~\eqref{eq:char}, $$L^k(\x,-1) = \sum_{\lambda\in\Lambda_k} \chi_\lambda^G x^\lambda,$$ 
 This provides us with a way of computing the generating function for the Gelfand character.
In particular, $G^k(\x)$ can be computed recursively.  

\begin{theorem}\label{main theorem 3}
We have $G^0(\x)=1$, $G^1(\x) = \dfrac{x_1}{(1-x_1^2)}$ and for $k \geq 2$,
$$
G^k(\x)=\frac{x_1}{(1-x_1^2)} G^{k-1}(\x; \hat{x}_1)+  \sum_{i=2}^k\frac{x_1x_i}{(1-x_1x_i)^2}G^{k-2}(\x; \hat{x}_1, \hat{x}_i). 
$$

\end{theorem}

\begin{proof}
First let us notice that in the statement of Theorem~\ref{thm descents}, the term $(1+t)$ is a factor of each summand involving $L^i(\x_V, t)$ for each $0\leq i\leq k-3$, so we can evaluate these terms at $t=-1$ and get zero. 
For $i=k-2$, the copy of $(1+t)$ in the numerator  will cancel with the copy of $(1+t)$ in the denominator, so we can evaluate at $t=-1$ and get:
\begin{align*}
\quad&\frac{(1-x_1^2) \prod_{j=2}^k(1-x_1x_j)^2}{(1-x_1)(1+x_1)\prod_{\ell=2}^k(1-x_1x_\ell)^2}\sum_{i=2}^kL^{k-2}(\x,-1;\hat{x}_1,\hat{x}_i)\frac{x_1x_i}{(1-x_1x_i)^2}\\
\quad&\quad\quad\quad=\sum_{i=2}^k\frac{x_1x_i}{(1-x_1x_i)^2}G^{k-2}(\x; \hat{x}_1, \hat{x}_i).
\end{align*}
Finally, let us consider the terms involving $L^{k-1}(\x_V,t)$ in Theorem~\ref{thm descents}:
$$
\frac{(1+tx_1^2)L^{k-1}(\x,t;\hat{x}_1)\prod_{j=2}^k(1+tx_1x_j)^2}{(1+t)(1-x_1)(1-tx_1)\prod_{\ell=2}^k(1-x_1x_\ell)(1-t^2x_1x_\ell)} -\frac{L^{k-1}(\x,t;\hat{x}_1)}{1+t}.
$$
Recall the notation 
$$
g(\x,t):=\prod_{j=2}^k(1+tx_1x_j)^2 - \prod_{j=2}^k(1-x_1x_j)(1-t^2x_1x_j)
$$
from the end of the previous section. Using $(1+tx_1^2) - (1-x_1)(1-tx_1) = (1+t)x_1$, the above expression can be written as
$$
\frac{(1+tx_1^2)g(\x,t)+x_1(1+t)\prod_{\ell=2}^k(1-x_1x_\ell)(1-t^2x_1x_\ell)}{(1+t)(1-x_1)(1-tx_1)\prod_{\ell=2}^k(1-x_1x_\ell)(1-t^2x_1x_\ell)}L^{k-1}(\x,t;\hat{x}_1).
$$
By Lemma \ref{lem:g}, $\frac{g(\x,t)}{(1+t)^2}$ is a polynomial. The substitution $t=-1$ will therefore annihilate the term involving $g(\x,t)$ and the expression will reduce to
$$
\frac{x_1\prod_{\ell=2}^k(1-x_1x_\ell)^2}{(1-x_1)(1+x_1)\prod_{\ell=2}^k(1-x_1x_\ell)^2}L^{k-1}(\x,-1;\hat{x}_1) = \frac{x_1}{(1-x_1^2)} G^{k-1}(\x; \hat{x}_1).
$$
This completes the proof.
\end{proof}

We can solve this recurrence to obtain the following generating function for the Gelfand character. Let $\mathbb{P}_k^2$ be the set of all partitions of the set $[k]=\{1,2,\ldots,k\}$ into subsets of sizes 1 and 2. For example, $\mathbb{P}_3^2$ contains
$$\{\{1\}, \{2\}, \{3\}\}, \quad \{\{1\}, \{2,3\}\}, \quad \{\{2\}, \{1,3\}\}, \quad \{\{3\}, \{1,2\}\},$$  and the set $\mathbb{P}_4^2$ contains 10 elements.
\begin{cor}\label{cor:Gelfand}
For $k\geq 1$, 
$$
G^k(\x) =\prod_{i=1}^n x_i\sum_{P\in\mathbb{P}_k^2} \prod_{\{i\} \in P} \frac{1}{(1-x_i^2)} \prod_{\{\ell,m\} \in P}\frac{1}{(1-x_\ell x_m)^2}.
$$
\end{cor}

For example, $$G^2(x,y) = xy\left[\frac{1}{(1-x^2)(1-y^2)} + \frac{1}{(1-xy)^2}\right]$$ and 
\begin{align*}
G^3(x,y,z) = xyz\bigg[&\frac{1}{(1-x^2)(1-y^2)(1-z^2)} + \frac{1}{(1-xy)^2(1-z^2)} \\
&+ \frac{1}{(1-xz)^2(1-y^2)} + \frac{1}{(1-yz)^2(1-x^2)}\bigg].
\end{align*}

Below are the first few terms of $G^k(\x)$ for $k\in \{1,2,3\}$ as computed from the formulas in Theorem \ref{main theorem 3}. 
\begin{align*}
G^1(x)   =& \, \,  x+x^3+x^5+x^7+ x^9 + x^{11} +x^{13} + x^{15}+ x^{17} + x^{19} +x^{21}+ x^{23} + x^{25} +\cdots \\
G^2(x,y) =& \, \,  2xy + xy^3 +2x^2y^2+ x^3y  + xy^5  + 4x^3y^3 +  x^5y  + xy^7 + x^3y^5+4x^4y^4   +\cdots \\
G^3(x,y,z) =& \, \,  4xyz + 2xyz^3 + 2xy^2z^2 +2x^2yz^2 + 2x^2y^2z + 2x^3yz   + 2xy^3z  + 4x^3yz^3+ \cdots \\
\end{align*}

\vspace{-18pt}

Notice that for each $k\geq1$, $G^k(\x)$ is symmetric in its $k$ variables. This follows from Corollary~\ref{cor:Gelfand}, but also from the observation that Equation \eqref{eq:char} holds for any ordering of the composition $\lambda$. Therefore, if we would like to compute $\chi^G_{(1,1,3)}$, we can take either the coefficient of $xyz^3$, $xy^3z$, or $x^3yz$ in $G^k(\x)$ as our answer. In each case, we find that $\chi^G_{(1,1,3)}=2$.

We can also obtain the following direct corollary of Corollary~\ref{cor:Gelfand}. 

\begin{cor}
Let $k\geq 1$ and $\lambda=(\lambda_1, \lambda_2,\ldots,\lambda_k)$. Then $\chi_\lambda^G$ is nonzero if and only there are an even number of copies of $i$ in $\lambda$ for each even number $i$. 
\end{cor}

That is, $\chi_\lambda^G\neq0$ if and only if the even components of the composition $\lambda$ come in pairs. We can see this is true since we start with the product of all variables (and so we start with each element appearing an odd number of times). Then any copy of $\frac{1}{(1-x_i^2)}$ will not change the parity of the exponents of any $x^i$. The only way to change  parity is with a copy of $\frac{1}{(1-x_\ell x_m)}$ in which case this change comes in pairs of equal exponents.





It remains to 
compare this method of computation with known methods (as in \cite{M-1937, N1940, R1997}). One thing that appears different about this computation is that the Gelfand character is computed on the $k$-th iteration for all $n$ so that $\lambda$ has $k$ parts. 



\section*{Acknowledgements} The authors would like to thank the University of Texas at Tyler's Office of Sponsored Research and Center for Excellence in Teaching and Learning for their support in conducting this research. The awards from these offices supported the research conducted for this paper by two faculty members, Kassie Archer and L.-K. Lauderdale, together with undergraduate student Marin King and graduate student Virginia Germany. We would also like to thank UT Tyler students Angela Gay, Thomas Lupo, and Francesca Rossi for their contributions to Proposition \ref{enumerate lambda descents} as part of a class project.

Finally, we would also like to thank the anonymous referees for their helpful comments and suggestions. 

\printbibliography

\end{document}